\documentclass{article}

\usepackage{amsmath}
\usepackage{amsthm}
\usepackage{amssymb}
\usepackage{bm}
\usepackage{bbm}
\usepackage{mathbbol}
\usepackage{latexsym}

\newtheorem{theorem}{Theorem}
\newtheorem{lemma}[theorem]{Lemma}
\newtheorem{corollary}[theorem]{Corollary}

\newtheorem{remark}[theorem]{Remark}

\newcommand{\ds}{\displaystyle}
\newcommand{\dd}{\mathrm{d}}
\newcommand{\E}{\mathrm{E}}
\newcommand{\one}{\mathbbm{1}}
\newcommand{\Z}{\mathbbm{Z}}
\newcommand{\M}{\mathcal{M}}

\allowdisplaybreaks

\begin{document}

\begin{center}{\Large\bf 
The Gated-Service M/GI/1 Queue with Single Vacations

\smallskip

and Its Application to Batch-Service Queues

}

\bigskip

Tetsuya Takine

\bigskip

Department of Information and Communications Technology, \\
Graduate School of Engineering, The University of Osaka, Japan
\end{center}

\bigskip

\noindent
\textbf{Abstract}

\noindent
In this paper, we consider the M/GI/1 queue with single vacations
under the gated service discipline. We obtain the probability
generating function of the stationary queue length, the
Laplace-Stieltjes transform of the system delay distribution in steady
state, and the joint transform of the busy cycle length and the number
of customers served in the busy cycle. Furthermore, as an application,
we consider a batch-service M/G/1 queue, where service times depend on
the number of customers in batch.

\bigskip


\bigskip


\section{Introduction}

We consider the M/G/1 queue with single vacations under the gated
service discipline.  Customers arrive according to a Poisson process
with rate $\lambda$.  Service times of customers are independent and
identically distributed (i.i.d.) according to a general distribution
$H(x)$ ($x \geq 0$). We assume $\rho=\lambda \E[H] < 1$, where $H$
denotes a generic random variable for service times.  Let
$h^*(s)=\E[\exp(-s H)]$ ($s > 0$). The server serves customers
according to the gated service discipline with single vacations.
\begin{narrow}\begin{itemize}
\item[(i)] 
When the system is empty and the server is idle, the server waits for
the next arrival and starts the service as soon as the next customer
arrives. After the service completion, the server takes a vacation.

\item[(ii-a)]
If there are $k$ ($k=1,2,\ldots$) waiting customers on return from a vacation,
the server serves those $k$ customers successively and takes another vacation.

\item[(ii-b)]
If the system is empty on return from a vacation, 
the server becomes idle. 
\end{itemize}\end{narrow}
We assume that vacation lengths are i.i.d.\ according to a general
distribution $V(x)$ ($x \geq 0$) and $\E[V] < \infty$, where $V$
denotes a generic random variable for vacation lengths. Let 
$v^*(s) = \E[\exp(-s V)]$ ($s > 0$).

M/GI/1 queues with gated vacations have been studied in the
literature, e.g., \cite{Takine92} and references therein.  Most of
them, however, consider so-called multiple-vacation queues, where if
the server finds the empty system on return from a vacation, the
server takes another vacation immediately, and to the best of our
knowledge, the above-mentioned single-vacation model has not been
considered. In \cite{Takagi91} and \cite{Tian06}, the following model
is called the gated-service M/G/1 queue with single vacations: The
server becomes idle after taking a vacation if no customers arrive
during the vacation, regardless of whether waiting customers (who
arrived during services of other customers) exist or not, and the
server starts to serve customers when the next arrival occurs.

In what follows, we first consider the queue length and system delay
distributions in steady state. We then consider the joint distribution
of the length of a busy cycle and the number of customers served in
the busy cycle. Finally, we provide some remarks on the MAP/GI/1 queue
with gated vacations.  For convenience, let $\{H_i\}_{i=1,2,\ldots}$
(resp.\ $\{V_i\}_{i=1,2,\ldots}$) denote a sequence of i.i.d.\ random
variables whose distribution function is given by $H(x)$
(resp.\ $V(x)$).

\section{Queue length and system delay distributions}

Let $N_{\lambda}(x)$ ($x \geq 0$) denote the number of Poisson
arrivals with rate $\lambda$ during an interval of length $x$.  We
then define the probability generating functions (PGF) $a_H^*(z)$ and
$a_V^*(z)$ as
\[
a_H^*(z) = \E[z^{N_{\lambda}(H)}] = h^*(\lambda-\lambda z),
\qquad
a_V^*(z) = \E[z^{N_{\lambda}(V)}] = v^*(\lambda-\lambda z).
\]
Let $L_{\rm B}$ and $L_{\rm E}$ denote the number of customers at the
beginning of a vacation and at the end of a vacation in steady state.
We define $\ell_{\rm B}^*(z)$ and $\ell_{\rm E}^*(z)$ as PGF's of
$L_{\rm B}$ and $L_{\rm E}$.
\[
\ell_{\rm B}^*(z) = \E[z^{L_{\rm B}}],
\qquad
\ell_{\rm E}^*(z) = \E[z^{L_{\rm E}}].
\]
By definition, we have
\begin{equation}
\ell_{\rm E}^*(z) = \ell_{\rm B}^*(z) a_V^*(z),
\qquad
\ell_{\rm B}^*(z) 
= 
\ell_{\rm E}^*( a_H^*(z) ) - \ell_{\rm E}^*(0)
+
\ell_{\rm E}^*(0) a_H^*(z).
\label{eq:ell_B(z)-ell_E(z)}
\end{equation}

Let $L$ denote the stationary queue length.  It is clear that the
model is a special case of the M/GI/1 queue with generalized vacations
\cite{Fuhrmann85}, so that the stochastic decomposition holds.
\begin{equation}
L = L_{\rm M/G/1} + L_{\rm V},
\label{eq:stochastic_decomposition}
\end{equation}
where $L_{\rm M/G/1}$ denote the stationary queue length in the
ordinary M/GI/1 queue characterized by the arrival rate $\lambda$ and
the service time distribution $H(x)$ and $L_{\rm V}$ denote the
conditional queue length in steady state, given that the server does
not serve customers. Note that $L_{\rm M/G/1}$ and $L_{\rm V}$ are
independent.  Because
\[
\Pr(\mbox{the server is idle} \mid 
\mbox{the server is not serving customers}) 
=
\frac{\ell_{\rm E}^*(0)\times 1/\lambda}
{\E[V]+\ell_{\rm E}^*(0) \times 1/\lambda},
\]
we have
\begin{equation}
L_{\rm V} = \begin{cases}
0, & \mbox{with probability } 
\ds\frac{\ell_{\rm E}^*(0)}{\lambda\E[V]+\ell_{\rm E}^*(0)},
\\[4mm]
L_{\rm B} + N_{\lambda}(\hat{V}),
 & \mbox{with probability } 
\ds\frac{\lambda \E[V]}{\lambda \E[V]+\ell_{\rm E}^*(0)},
\end{cases}
\label{eq:L_V}
\end{equation}
where $\hat{V}$ denotes the equilibrium random variable of $V$, which is 
independent of $L_{\rm B}$.

We define $\ell^*(z)$ as the PGF of the stationary queue length $L$. 
\[
\ell^*(z) = \E[z^L], 
\quad
|z| < 1.
\]
It then follows from (\ref{eq:stochastic_decomposition}) and 
(\ref{eq:L_V}) that 
\begin{equation}
\ell^*(z) 
=
\ell_{\rm M/G/1}^*(z)
\bigg[
\frac{\ell_{\rm E}^*(0)}{\lambda\E[V]+\ell_{\rm E}^*(0)}
+
\frac{\lambda\E[V]}{\lambda\E[V]+\ell_{\rm E}^*(0)}
\ell_{\rm B}^*(z) 
\frac{1-a_V^*(z)}{\lambda\E[V](1-z)}
\bigg],
\quad
|z| < 1,
\label{eq:ell^*(z)-decomposition}
\end{equation}
where $\ell_{\rm M/G/1}^*(z)$ denotes the PGF of the stationary queue
length in the ordinary M/GI/1 queue.
\begin{equation}
\ell_{\rm M/G/1}^*(z)
=
\frac{(1-\rho) (z-1) a_H^*(z)}{z - a_H^*(z)},
\quad
|z| < 1.
\label{eq:ell_M/G/1}
\end{equation}
We observe from (\ref{eq:ell_B(z)-ell_E(z)}) and
(\ref{eq:ell^*(z)-decomposition}) that if $\ell_{\rm E}^*(z)$ is
obtained, $\ell^*(z)$ is determined completely.  In what follows, we
derive $\ell_{\rm E}^*(z)$.

Let $Q_n$ ($n=0,1,\ldots$) denote the number of customers at the end
of the $n$th vacation. It is clear that $\{Q_n\}$ forms a
discrete-time Markov chain on $\Z^+=\{0,1,\ldots\}$. 
Because 
\begin{align*}
\Pr( Q_{n+1} = j \mid Q_n=0) 
&>
\Pr(N_{\lambda}(H)=0)\Pr(N_{\lambda}(V)=j)
>0,
\\
\Pr( Q_{n+1} = j \mid Q_n=i) 
&>
\Pr(N_{\lambda}(H_1+H_2+\cdots+H_i)=0)\Pr(N_{\lambda}(V)=j)
>0,
\quad
i=1,2,\ldots,
\end{align*}
for all $j \in \Z^+$, the Markov chain $\{Q_n\}$ is irreducible and
aperiodic. Furthermore, because $\rho < 1$, $\E[V] < \infty$, and
for $j=1,2,\ldots$, 
\[
\E[Q_{n+1} \mid Q_n = j] < j\ \Leftrightarrow\ 
\lambda \big(j \E[H] + \E[V]\big) < j\ 
\Leftrightarrow\ j > \frac{\lambda \E[V]}{1-\rho},
\]
there exists $j^*$ such that $\E[Q_{n+1} \mid Q_n = j] < j$ for all $j
\geq j^*$, so that $\{Q_n\}$ is positive recurrent.  Therefore, if
$\rho < 1$ and $\E[V] < \infty$, the limiting distribution of
$\{Q_n\}$ exists, i.e.,
\begin{equation}
\lim_{n \to \infty} \Pr(Q_n = k) = \Pr(L_{\rm E} =k),
\quad
k=0,1,\ldots.
\label{eq:limiting_dist}
\end{equation}
Note that the limiting distribution is independent of the initial
state $Q_0$. 

We define $q_n^*(z)$ ($n=0,1,\ldots$) as the PGF of $Q_n$.
\[
q_n^*(z) = \E[z^{Q_n}],
\quad
n=0,1,\ldots.
\]

\begin{theorem}\label{theorem:q_n^*(z)}
The PGF $q_n^*(z)$ ($n=1,2,\ldots$) satisfies
\begin{equation}
q_n^*(z) 
= 
q_0^*(\psi_n^H(z)) \prod_{k=1}^n \psi_k^V(z)
-
\sum_{j=1}^n 
q_{n-j}^*(0) \big( 1 - \psi_j^H(z) \big) \prod_{k=1}^j \psi_k^V(z),
\quad 
n=1,2,\ldots,
\label{eq:theorem:q_n^*(z)}
\end{equation}
where $\psi_n^H(z)$ and $\psi_n^V(z)$ ($n=1,2,\ldots$) are defined
recursively as
\begin{align}
\psi_1^H(z) 
&=
a_H^*(z),
\qquad
\psi_n^H(z) 
= 
a_H^*\big(\psi_{n-1}^H(z)\big)
=
\psi_{n-1}^H\big(a_H^*(z)\big),
\quad 
n=2,3,\ldots,
\label{eq:psi^H}
\\
\psi_1^V(z) 
&=
a_V^*(z), 
\qquad 
\psi_n^V(z) = a_V^*\big(\psi_{n-1}^H(z)\big),
\quad 
n=2,3,\ldots.
\label{eq:psi^V}
\end{align}
\end{theorem}

\begin{remark}
If we model customers served in a busy period of the ordinary M/G/1
queue by a Galton-Watson branching process \cite{Athreya72},
$\psi_n^H(z)$ denote the number of customers in the $n$th
generation. Similarly, $\psi_n^V(z)$ can be related to the number of
customers in the $n$th generation in a delayed busy period, where the
service time of the first customer follows the distribution $V(x)$.
\end{remark}

\begin{proof}
We prove Theorem \ref{theorem:q_n^*(z)} by induction. 
By definition, we have
\[
q_1^*(z) 
= 
\big(q_0^*(a_H^*(z)) - q_0^*(0) + q_0^*(0) a_H^*(z)\big) a_V^*(z)
=
q_0^*(\psi_1^H(z)) \psi_1^V(z) 
-  
q_0^*(0) \big(1 - \psi_1^H(z)\big) \psi_1^V(z),
\]
so that Theorem \ref{theorem:q_n^*(z)} holds for $n=1$.
Suppose Theorem \ref{theorem:q_n^*(z)} holds for some $n$ 
($n=1,2,\ldots$). It then follows that
\begin{align}
q_{n+1}^*(z) 
&= 
\big(q_n^*(a_H^*(z)) - q_n^*(0) + q_n^*(0) a_H^*(z)\big) a_V^*(z)
\label{eq:q_n=q_{n-1}}
\\
&=
\bigg[
q_0^*(\psi_n^H(a_H^*(z))) \prod_{k=1}^n \psi_k^V(a_H^*(z)) 
-
\sum_{j=1}^n 
q_{n-j}^*(0) \big( 1 - \psi_j^H(a_H^*(z)) \big) 
\prod_{k=1}^j \psi_k^V(a_H^*(z)) 
\bigg]
a_V^*(z)
\notag
\\
& \qquad {} 
-
q_n^*(0) \big( 1 - a_H^*(z)\big) a_V^*(z)
\notag
\\
&=
q_0^*(\psi_{n+1}^H(z)) \bigg(\prod_{k=1}^n \psi_{k+1}^V(z) \bigg)) \psi_1^V(z)
-
\sum_{j=1}^n 
q_{n-j}^*(0) \big( 1 - \psi_{j+1}^H(z) \big) 
\bigg(\prod_{k=1}^j \psi_{k+1}^V(z) \bigg) \psi_1^V(z)
\notag
\\
& \qquad {} 
-
q_n^*(0) \big( 1 - \psi_1^H(z)\big) \psi_1^V(z)
\notag
\\
&=
q_0^*(\psi_{n+1}^H(z)) \prod_{k=1}^{n+1} \psi_{k+1}^V(z)
-
\sum_{j=2}^{n+1}
q_{n+1-j}^*(0) \big( 1 - \psi_j^H(z) \big) 
\prod_{k=1}^j \psi_k^V(z) 
-
q_n^*(0) \big( 1 - \psi_1^H(z)\big) \psi_1^V(z),
\notag
\end{align}
which shows that 
Theorem \ref{theorem:q_n^*(z)} holds for $n+1$.
\end{proof}

It follows from (\ref{eq:limiting_dist}) that 
\begin{equation}
\lim_{n \to \infty} q_n^*(z) = \ell_E^*(z),
\quad
|z| < 1.
\label{eq:limiting_distribution}
\end{equation}
To obtain $\ell_E^*(z)$, we use the the final value theorem of the
$Z$-transform.
\begin{lemma}[Final value theorem of the $Z$-transform]\label{lemma:FVT}
If $\{a_n\}_{n=1,2,\ldots}$ has a finite limit $a^*$, then 
\[
a^* 
= 
\lim_{n \to \infty} a_n 
=
\lim_{\omega \to 1-} (1-\omega) \sum_{n=1}^{\infty} a_n \omega^n
\]
\end{lemma}

For completeness, we provide the proof of Lemma \ref{lemma:FVT} in
Appendix \ref{appendix:proof_lemma:FVT}.

\begin{theorem}\label{theorem:ell_E^*(z)}
If $\rho =\lambda \E[H] < 1$ and $\E[V] < \infty$, $\ell_E^*(z)$ is given
by
\begin{equation}
\ell_E^*(z)
=
\prod_{k=1}^{\infty} \psi_k^V(z)
-
\ell_E^*(0)
\sum_{n=1}^{\infty}
\big( 1 - \psi_n^H(z) \big) \prod_{k=1}^n \psi_k^V(z),
\label{eq:ell_E^*(z)}
\end{equation}
where 
\begin{equation}
\ell_E^*(0)
=
\ds\frac{\ds\prod_{k=1}^{\infty} \psi_k^V(0)}
{1 + \ds\sum_{n=1}^{\infty}
\big( 1 - \psi_n^H(0) \big)
\ds\prod_{k=1}^n \psi_k^V(0)}.
\label{eq:ell_E^*(0)}
\end{equation}
\end{theorem}

\begin{proof}
We apply Lemma \ref{lemma:FVT} to (\ref{eq:theorem:q_n^*(z)}).
Specifically, multiplying both sides of (\ref{eq:theorem:q_n^*(z)}) by
$(1-\omega) \omega^n$ ($|\omega| < 1$) and summing up both sides from
$n=1$ to $\infty$ yield
\begin{align*}
(1-\omega) \sum_{n=1}^{\infty}
q_n^*(z) \omega^n
&=
(1-\omega) \sum_{n=1}^{\infty}
\bigg(q_0^*\big(\psi_n^H(z)\big) \prod_{k=1}^n \psi_k^V(z)\bigg)
\omega^n
\\
&\qquad\qquad\qquad {} 
-
(1-\omega) \sum_{n=1}^{\infty}
q_n^*(0) \omega^n
\times
\sum_{n=1}^{\infty}
\bigg(
\big( 1 - \psi_n^H(z) \big) \prod_{k=1}^n \psi_k^V(z)
\bigg)
\omega^n.
\end{align*}
Taking the limit $\omega \to 1-$, noting
(\ref{eq:limiting_distribution}), and applying Lemma \ref{lemma:FVT}
to the above equation, we obtain
\begin{equation}
\ell_E^*(z)
=
\lim_{\omega \to 1-}
(1-\omega) \sum_{n=1}^{\infty}
\bigg(q_0^*\big(\psi_n^H(z)\big) \prod_{k=1}^n \psi_k^V(z)\bigg)
\omega^n
-
\ell_E^*(0)
\times
\lim_{\omega \to 1-}
\sum_{n=1}^{\infty}
\bigg(
\big( 1 - \psi_n^H(z) \big) \prod_{k=1}^n \psi_k^V(z)
\bigg)
\omega^n.
\label{eq:ell_E-limit}
\end{equation}
Note here that 
\[
\lim_{z \to 1-} q_0^*(z)=1,
\qquad
\lim_{n \to \infty} \psi_n^H(z)=1,
\qquad
0 < \psi_k^V(z) < \psi_k^V(1) = 1,
\quad
z \in (0,1),
\]
and $0 < \prod_{k=1}^n \psi_k^V(z) < 1$ ($z \in (0,1)$) is a decreasing
function of $n$. It then follows from Lemma \ref{lemma:FVT} that 
\begin{equation}
\lim_{\omega \to 1-}
(1-\omega) \sum_{n=1}^{\infty}
\bigg(q_0^*\big(\psi_n^H(z)\big) \prod_{k=1}^n \psi_k^V(z)\bigg)
\omega^n
=
\lim_{n \to \infty}
\bigg(q_0^*\big(\psi_n^H(z)\big) \prod_{k=1}^n \psi_k^V(z)\bigg)
=
\prod_{k=1}^{\infty} \psi_k^V(z).
\label{eq:1st-limit}
\end{equation}
On the other hand, 
\begin{equation}
\lim_{\omega \to 1-}
\sum_{n=1}^{\infty}
\bigg(
\big( 1 - \psi_n^H(z) \big) \prod_{k=1}^n \psi_k^V(z)
\bigg)
\omega^n
=
\sum_{n=1}^{\infty}
\bigg(
\big( 1 - \psi_n^H(z) \big) \prod_{k=1}^n \psi_k^V(z)
\bigg),
\label{eq:2nd-limit}
\end{equation}
and we can show the finiteness of the infinite sum on the right-hand
side of (\ref{eq:2nd-limit}) as follows.  Note first that
\[
\sum_{n=1}^{\infty}
\bigg(
\big( 1 - \psi_n^H(z) \big) \prod_{k=1}^n \psi_k^V(z)
\bigg)
\leq
\sum_{n=1}^{\infty}
\big(1 - \psi_n^H(z)\big)
\leq 
\sum_{n=1}^{\infty}
\big(1-\psi_n^H(0)\big),
\]
where $1-\psi_n^H(0)$ can be regarded as the probability that the
number of generations in a busy period of the ordinary M/G/1 queue is
greater than $n$. We thus have
\begin{align*}
1 + \sum_{n=1}^{\infty} \big(1-\psi_n^H(0)\big)
&=
\mbox{the mean number of generation in a busy period of the ordinary M/G/1 queue}
\\
&\leq 
\mbox{the mean number of customers served in a busy period of the ordinary 
M/G/1 queue}
\\
&= 
\frac{1}{1-\rho}.
\end{align*}
Substituting (\ref{eq:1st-limit}) and (\ref{eq:2nd-limit}) into
(\ref{eq:ell_E-limit}), we obtain (\ref{eq:ell_E^*(z)}).
Finally, (\ref{eq:ell_E^*(0)}) is obtained by substituting $z=0$ 
in (\ref{eq:ell_E^*(z)}) and rearranging terms.
\end{proof}

In summary, we have the following theorem.

\begin{theorem}\label{theorem:L-summary}
If $\rho < 1$ and $\E[V] < \infty$, the system is stable and the PGF
$\ell^*(z)$ of the stationary queue length distribution is given by
\begin{equation}
\ell^*(z) 
=
\ell_{\rm M/G/1}^*(z)
\bigg[
\frac{\ell_{\rm E}^*(0)}{\lambda\E[V]+\ell_{\rm E}^*(0)}
+
\frac{\lambda\E[V]}{\lambda\E[V]+\ell_{\rm E}^*(0)}
\cdot
\frac{\ell_{\rm E}^*(z)}{a_V^*(z)}
\cdot
\frac{1-a_V^*(z)}{\lambda\E[V](1-z)}
\bigg],
\quad
|z| < 1,
\label{eq:theorem-ell^*(z)}
\end{equation}
where $\ell_{\rm M/G/1}^*(z)$, $\ell_E^*(z)$, and $\ell_E^*(0)$ are
given in (\ref{eq:ell_M/G/1}), (\ref{eq:ell_E^*(z)}), and
(\ref{eq:ell_E^*(0)}), respectively.
\end{theorem}

\begin{proof}
The theorem follows from (\ref{eq:ell_B(z)-ell_E(z)}) and 
(\ref{eq:ell^*(z)-decomposition}).
\end{proof}

We now consider the factorial moments of the stationary queue length $L$.
We define $\overline{L}^{(n)}$ as the $n$th factorial moment of $L$.
It then follows that 
\[
\overline{L}^{(n)}
=
\E[L(L-1)\cdots(L-n+1)]
=
\lim_{z \to 1-} \frac{\dd^n}{\dd z^n} \ell^*(z),
\quad
n=1,2,\ldots.
\]
To compute $\overline{L}^{(n)}$, we may rewrite 
(\ref{eq:theorem-ell^*(z)}) as
\begin{equation}
\ell^*(z) a_V^*(z)
=
\ell_{\rm M/G/1}^*(z)
\bigg[
\frac{\ell_{\rm E}^*(0)}{\lambda\E[V]+\ell_{\rm E}^*(0)} \cdot a_V^*(z)
+
\frac{\lambda\E[V]}{\lambda\E[V]+\ell_{\rm E}^*(0)}
\cdot
\frac{1-a_V^*(z)}{\lambda\E[V](1-z)}
\cdot
\ell_{\rm E}^*(z)
\bigg].
\label{eq:ell^*(z)a_V^*(z)}
\end{equation}
We then define $\overline{L}_{\rm M/G/1}^{(n)}$, $\overline{A}_H^{(n)}$, 
and $\overline{A}_V^{(n)}$ as
\[
\overline{L}_{\rm M/G/1}^{(n)}
=
\lim_{z \to 1-} \frac{\dd^n}{\dd z^n} \ell_{\rm M/G/1}^*(z),
\qquad
\overline{A}_H^{(n)}
=
\lim_{z \to 1-} \frac{\dd^n}{\dd z^n} a_H^*(z),
\qquad
\overline{A}_V^{(n)}
=
\lim_{z \to 1-} \frac{\dd^n}{\dd z^n} a_V^*(z),
\quad
n=1,2,\ldots.
\]
It is easy to see that for $n=1,2,\ldots$,
\[
\overline{A}_H^{(n)}
=
\lambda^n \E[H^n],
\qquad
\overline{A}_V^{(n)}
=
\lambda^n \E[V^n],
\qquad
\overline{L}_{\rm M/G/1}^{(n)} 
= 
\sum_{k=0}^{n-1} \binom{n+1}{k}
\frac{\overline{A}_H^{(n+1-k)}
\overline{L}_{\rm M/G/1}^{(k)}}{(n+1) (1-\rho)}
+ \overline{A}_H^{(n)}, 
\]
and
\[
\lim_{z \to 1-} \frac{\dd^n}{\dd z^n} 
\frac{\lambda\E[V]}{\lambda\E[V]+\ell_{\rm E}^*(0)}
\cdot
\frac{1-a_V^*(z)}{(\lambda-\lambda z)\E[V]}
=
\frac{1}{\lambda\E[V]+\ell_{\rm E}^*(0)}
\cdot
\frac{\overline{A}_V^{(n+1)}}{n+1},
\]
where $\overline{L}_{\rm M/G/1}^{(0)} = 1$.
It then follows from (\ref{eq:ell^*(z)a_V^*(z)}) that 
\begin{align*}
\overline{L}^{(n)}
&=
\sum_{k=0}^n \overline{L}_{\rm M/G/1}^{(n-k)}
\bigg[
\frac{\ell_E^*(0)}{\lambda\E[V]+\ell_E^*(0)}
\lambda^k \E[V^k]
+
\frac{\lambda\E[V]}{\lambda\E[V]+\ell_E^*(0)}
\sum_{i=0}^k 
\binom{k}{i} \frac{\lambda^{k+1}\E[V^{k+1}]}{(k+1)\lambda\E[V]}
\overline{L}_E^{(k-i)}
\bigg]
\\
&\qquad\qquad {} 
-
\sum_{k=0}^{n-1}
\binom{n}{k} \overline{L}^{(k)} \lambda^{n-k} \E[V^{n-k}],
\end{align*}
where 
\[
\overline{L}_E^{(n)}
=
\lim_{z \to 1-} \frac{\dd^n}{\dd z^n} \ell_E^*(z),
\quad
n=1,2,\ldots.
\]
Therefore, the computation of $\overline{L}^{(n)}$ is reduced to 
the computation of $\overline{L}_E^{(n)}$.

It follows from (\ref{eq:ell_B(z)-ell_E(z)}) that
\[
\ell_{\rm E}^*(z) 
= 
\ell_{\rm E}^*( a_H^*(z) ) a_V^*(z)
-
\ell_{\rm E}^*(0)
\big[1 - \ell_{\rm E}^*(0) a_H^*(z)\big] a_V^*(z).
\]
Because the general formula for the high-order derivatives
of a composite function is not available, we have to compute 
$\overline{L}_E^{(n)}$ manually. For example,
\[
\overline{L}_E^{(1)}
=
\rho \overline{L}_E^{(1)} + \lambda \E[V] + \rho \ell_E^*(0),
\qquad
\overline{L}_E^{(2)}
=
\rho^2 \overline{L}_E^{(2)} 
+ 
\big(\overline{A}_H^{(2)} + 2 \rho \lambda\E[V]\big)
\big(\overline{L}_E^{(1)} + \ell_E^*(0)\big)
+\overline{A}_V^{(2)},
\]
so that 
\[
\overline{L}_E^{(1)}
=
\frac{\lambda \E[V] + \rho \ell_E^*(0)}{1-\rho},
\qquad
\overline{L}_E^{(2)}
=
\frac{\big(\lambda^2\E[H^2] + 2 \rho \lambda\E[V]\big)
\big(\overline{L}_E^{(1)} + \ell_E^*(0)\big)
+\lambda^2 \E[V^2]}
{1-\rho^2}.
\]
Using such results, we can compute $\overline{L}^{(n)}$
($n=1,2,\ldots$).  In particular,
\begin{align*}
\overline{L}^{(1)}
&=
\overline{L}_{\rm M/G/1}^{(1)}
+
\frac{\lambda\E[V]}{\lambda\E[V]+\ell_{\rm E}^*(0)}
\big( \overline{L}_E^{(1)} - \lambda \E[V] \big)
+
\frac{\lambda\E[V]}{\lambda\E[V]+\ell_{\rm E}^*(0)}
\cdot
\frac{\lambda^2 \E[V^2]}{2\lambda\E[V]}
\\
&=
\frac{\lambda^2\E[H^2]}{2(1-\rho)} + \rho
+
\frac{\rho}{1-\rho} \lambda\E[V]
+
\frac{\lambda\E[V]}{\lambda\E[V]+\ell_{\rm E}^*(0)}
\cdot
\frac{\lambda^2 \E[V^2]}{2\lambda\E[V]}.
\end{align*}

\begin{remark}
In the stationary M/G/1 queue with multiple vacations under the gated
service discipline, the mean queue length $\E[L_{\rm MV}]$ is given by
\cite{Takagi91,Tian06}
\[
\E[L_{\rm MV}]
=
\frac{\lambda^2\E[H^2]}{2(1-\rho)} + \rho
+
\frac{\rho}{1-\rho} \lambda\E[V]
+
\frac{\lambda^2 \E[V^2]}{2\lambda\E[V]}.
\]
We thus have
\[
\E[L_{\rm MV}] - \E[L]
=
\frac{\ell_E^*(0)}{\lambda\E[V]+\ell_{\rm E}^*(0)}
\cdot
\frac{\lambda^2 \E[V^2]}{2\lambda\E[V]} 
> 0.
\]
\end{remark}

Next, we consider the system delay under the assumption that customers
are served on a FIFO basis.  Let $D$ denote the system delay of a
randomly chosen tagged customer in steady state and let
$d^*(s)=\E[\exp(-s D)]$ ($s > 0)$.

\begin{corollary}
If $\rho < 1$ and $\E[V] < \infty$, the LST $d^*(s)$ of the system delay $D$
in the stationary FIFO system is given by
\begin{equation}
d^*(s) 
=
\frac{(1-\rho)s h^*(s)}{s - \lambda + \lambda h^*(s)}
\bigg[
\frac{\ell_{\rm E}^*(0)}{\lambda\E[V]+\ell_{\rm E}^*(0)}
+
\frac{\lambda\E[V]}{\lambda\E[V]+\ell_{\rm E}^*(0)}
\cdot
\frac{d_{\rm E}^*(s)}{v^*(s)}
\cdot
\frac{1-v^*(s)}{\E[V]s}
\bigg],
\label{eq:d^*(s)}
\end{equation}
where 
\[
d_{\rm E}^*(s)
=
\prod_{k=1}^{\infty} \zeta_k^V(s)
-
\ell_E^*(0)
\sum_{n=1}^{\infty}
\big( 1 - \zeta_n^H(s) \big) \prod_{k=1}^n \zeta_k^V(s),
\]
and $\zeta_k^H(s)$ and $\zeta_k^V(s)$ are determined recursively by
\begin{align*}
\zeta_1^H(s) 
&=
h^*(s)
\qquad
\zeta_n^H(s) 
= 
\psi_{n-1}^H\big(h^*(s)\big),
\quad 
n=2,3,\ldots
\\
\zeta_1^V(s) 
&=
v^*(s), 
\qquad 
\zeta_n^V(s) = v^*\big(\lambda-\lambda \zeta_{n-1}^H(s)\big),
\quad 
n=2,3,\ldots.
\end{align*}
The $n$th ($n=1,2,\ldots$) moment $\overline{D}^{(n)}$ of $D$ is given by
\begin{equation}
\overline{D}^{(n)}
=
\frac{\overline{L}^{(n)}}{\lambda^n},
\quad
n=1,2,\ldots.
\label{eq:overline{D}}
\end{equation}
\end{corollary}

\begin{proof}
Note that the
system delay of the tagged customer is independent of subsequent
arrivals of the tagged customer. Therefore, the distributional form of
Little's law \cite{Keilson88} holds, i.e., $\ell^*(z) =
d^*(\lambda-\lambda z)$, or equivalently, 
\begin{equation}
d^*(s) = \ell^*\big((\lambda-s)/\lambda\big),
\label{eq:distributional}
\end{equation}
from which (\ref{eq:overline{D}}) follows.
Eq.\ (\ref{eq:d^*(s)}) follows from (\ref{eq:distributional}) and 
Theorems \ref{theorem:q_n^*(z)}, 
\ref{theorem:ell_E^*(z)}, and \ref{theorem:L-summary}.
\end{proof}

Before closing this section, we remark on the numerical computation of
$\ell_E^*(0)$. It follows from (\ref{eq:ell_E^*(0)}) that
\[
\ell_E^*(0)
= 
\frac{\ds\lim_{n \to \infty} \Psi_n^V}
{1 + \ds\sum_{n=1}^{\infty} (1 - \psi_n^H(0)) \Psi_n^V},
\]
where 
\[
\Psi_1^V=\psi_1^V(0) =v^*(\lambda)
\qquad
\Psi_n^V = \prod_{k=1}^n \psi_k^V(0) = \psi_n^V(0) \Psi_{n-1}^V,
\quad 
n=2,3,\ldots.
\]
Noting (\ref{eq:psi^H}) and (\ref{eq:psi^V}), we have
\begin{gather*}
\psi_1^H(0) = h^*(\lambda),
\qquad
\psi_n^H(0) = h^* \big(\lambda - \lambda \psi_{n-1}^H(0) \big),
\quad
n=2,3,\ldots
\\
\psi_1^V(0) = v^*(\lambda),
\qquad
\psi_n^V(0) = v^*\big( \lambda - \lambda \psi_{n-1}^H(0) \big),
\quad
n=2,3,\ldots
\end{gather*}
Because $\{\psi_n^H(0)\}$ and $\{\psi_n^V(0)\}$ are monotonously
increasing sequences converging to one, we may compute $\ell_E^*(0)$
by
\[
\ell_E^*(0)
\approx 
\frac{\Psi_{n^*}^V}{1 + \ds\sum_{n=1}^{n^*} (1 - \psi_n^H(0)) \Psi_n^V},
\]
where $n^*$ might be the minimum integer such that 
\[
1-\psi_{n^*}^H(0) < 10^{-16},
\qquad
1 - \psi_{n^*}^V(0) < 10^{-16}.
\]

\section{Busy Cycle}\label{section:busy_cycle}

The queue length process after the start of an idle period is
independent of the past history and the lengths of idle periods are
i.i.d.\ according to an exponential distribution with parameter
$\lambda$. We define a busy cycle as an interval from the end of an
idle period to the start of the next idle period. Note that the system
is in an idle period and in a busy period alternately, and we can
regard this as an alternating renewal process.

For convenience, we assume that an arrival occurs at time 0, which
breaks an idle period, and we define $T^*$ as 
\[
T^* = \mbox{the time when the busy cycle ends after time 0}.
\]
We are interested in the joint distribution of $N^*$, $T^*$, and
$M^*$, where
\begin{align*}
N^* &= \mbox{the total number of customers served in $(0,T^*]$},
\\
M^* &= \mbox{the number of vacations before time $T^*$}.
\end{align*}
We define $\theta_n^*(z,\omega)$ as
\begin{equation}
\theta_n(z,\omega) = \E\big[ z^{N^*} e^{-\omega T^*} \one_{\{M^*=n\}}\big],
\quad
n=1,2,\ldots.
\label{eq:theta_n-def}
\end{equation}

To derive $\theta_n^*(z,\omega)$ ($n=2,3,\ldots$), we consider the
following delayed regenerative process.  During the interval
$(0,T^*]$, the system is operated according to the single vacation
policy, i.e., the service of the single customer arriving at time 0
starts. By definition, a vacation ends with the empty system at time
$T^*$. At this moment, we switch the operation policy from the single
vacation to the multiple vacation.  Specifically, if there are no
waiting customers at the end of a vacation, the server takes another
vacation immediately. We note that the regenerative cycle in the
gated-service M/G/1 queue with multiple vacations is studied in
\cite{Takine92}.

We distinguish customers in the above delayed regenerative process,
using ancestral lines in the Galton-Watson branching process
\cite{Fuhrmann85}. Specifically, the first arriving customer at
time 0 is the ancestor (i.e., the 0th generation) of ancestral line 0,
and customers arriving during services of customers in the $(k-1)$st
($k=1,2,\ldots$) generation of the ancestral line 0 are called the
$k$th generation of ancestral line 0.
Let $N_k$ and $B_k^H$ denote the number of customers in the $k$th
generation and the sum of service times of $N_0+N_1+\cdots+N_k$
customers in the 0th to $k$th generations. Note here that $N_0=1$ and
if $N_k=0$, then $N_{k+1}=0$.  We define $\psi_k^H(z,\omega,\eta)$ as
\[
\psi_k^H(z,\omega,\eta)
=
\E\big[ z^{N_0+N_1+\cdots+N_{k-1}}
e^{-\omega B_{k-1}^H}  \eta^{N_k} \big],
\quad
k=1,2,\ldots,
\]
It then follows that
\[
\psi_1^H(z,\omega,\eta)
=
z h^*(\omega+\lambda-\lambda \eta),
\qquad
\psi_k^H(z,\omega, \eta)
=
z h^*\big(\omega + \lambda - \lambda \phi_{k-1}^H(z,\omega,\eta)\big),
\quad
k=2,3,\ldots.
\]

Similarly, customers who arrive during the $n$th ($n=1,2,\ldots$)
vacation after time 0 are defined as the first generation of the
ancestral line $n$, and customers arriving during services of
customers in the $(k-1)$st ($k=2,3,\ldots$) generation of ancestral
line $n$ are called the $k$th generation of ancestral line $n$.  Let
$V_0$ denote the vacation length during which customers in the first
generation arrive and $B_k^V$ ($k=2,3,\ldots$) denote the sum of
$V_0$ and all service times of customers in the first to $k$th
generations.  We then define $\psi_k^V(z,\omega,\eta)$ as
\[
\psi_1^V(z,\omega,\eta)
=
\E\big[ e^{-\omega V_0}  \eta^{N_1} \big],
\qquad
\psi_k^V(z,\omega,\eta)
=
\E\big[ z^{N_1+\cdots+N_{k-1}}
e^{-\omega B_{k-1}^V}  \eta^{N_k} \big],
\quad
k=2,3,\ldots.
\]
It then follows that 
\[
\psi_1^V(z,\omega,\eta)
=
v^*(\omega+\lambda-\lambda \eta),
\qquad
\psi_k^V(z,\omega, \eta)
=
v^*\big(\omega + \lambda - \lambda \phi_{k-1}^H(z,\omega,\eta)\big),
\quad
k=2,3,\ldots.
\]

\begin{remark}
The branching processes of ancestral lines $0,1,2,\ldots,n$ ($n \geq
1$) are mutually independent owing to the independent increment
property of the Poisson process, and except for the branching process
of ancestral line 0, all branching processes are stochastically
identical.
\end{remark}

\begin{theorem}
The joint transform $\theta_n(z,\omega)$ ($n=1,2,\ldots$) in
(\ref{eq:theta_n-def}) is determined recursively by
\begin{align}
\theta_1(z,\omega) 
&=
z h^*(\omega+\lambda) v^*(\omega+\lambda),
\label{eq:theta_1}
\\
\theta_n(z,\omega)
&=
\psi_n^H(z,\omega,0) \prod_{m=1}^n \psi_m^V(z,\omega,0)
-
\sum_{k=1}^{n-1} \theta_k(z,\omega)
\prod_{m=1}^{n-k} \psi_m^V(z,\omega,0),
\quad
n=2,3,\ldots.
\label{eq:theta_n}
\end{align}
\end{theorem}

\begin{proof}
By definition, we have 
\begin{equation}
\theta_1(z,\omega)
=
\psi_1^H(z,\omega,0) \psi_1^V(z,\omega,0),
\label{eq:theta_1-org}
\end{equation}
from which (\ref{eq:theta_1}) follows.  Next we consider
$\theta_n(z,\omega)$ ($n=2,3,\ldots$).  We note that
$\psi_k^H(z,\omega, 0)$ represents the joint transform of the total
number of customers and the sum of service times of all customers in
ancestral line 0 when $N_k=0$.  Similarly, $\psi_k^V(z,\omega, 0)$
represents the joint transform of the total number of customers and
the sum of the vacation length and service times of all customers, in
ancestral line $m$ ($m \geq 1$) when $N_k=0$.  Let $T_n$
($n=1,2,\ldots$) denote the time when the $n$th vacation ends in the
delayed regenerative process, and let $S_n$ denote the total number of
customers served in $(0,T_n]$. Moreover, we define $R_n$
  ($n=1,2,\ldots$) as
\[
R_n
=
\begin{cases}
0, & \mbox{the system is empty at time $T_n$}, \\
1, & \mbox{otherwise}.
\end{cases}
\]
At time $T_n$, waiting
customers in the system can be classified as follows.  Customers in
the $n$th generation of ancestral line 0, and customers in the
$(n+1-m)$th ($m=1,2,\ldots,n$) generation of ancestral line $m$. Based
on this observation, we obtain
\[
\E\big[ z^{N^{\rm total}(T_n)} e^{-\omega T_n} \one_{\{R_n=0\}} \big]
=
\psi_n^H(z,\omega,0) \prod_{m=1}^n \psi_m^V(z,\omega,0),
\quad
n=1,2,\ldots,
\]
where $N^{\rm total}(t)$ denote the total number of customers served
in $(0,t]$. Therefore, we obtain the following discrete-time renewal
equation.
\[
\psi_n^H(z,\omega,0) \prod_{m=1}^n \psi_m^V(z,\omega,0)
=
\theta_n(z,\omega)
+
\sum_{k=1}^{n-1} \theta_k(z,\omega)
\prod_{m=1}^{n-k} \psi_m^V(z,\omega,0),
\quad
n=2,3,\ldots,
\]
from which (\ref{eq:theta_n}) follows.
\end{proof}

We define $\theta^*(z,\omega)$ as 
\[
\theta^*(z,\omega)
=
\E\big[ z^{N^*} e^{-\omega T^*}\big].
\]

\begin{corollary}
If $\rho < 1$ and $\E[V] < \infty$, the joint transform
$\theta^*(z,\omega)$ of $N^*$ and $T^*$ is given by
\begin{equation}
\theta^*(z,\omega)
=
\frac{\ds\sum_{n=1}^{\infty}
\psi_n^H(z,\omega,0) \ds\prod_{m=1}^n \psi_m^V(z,\omega,0)}
{1 + 
\ds\sum_{n=1}^{\infty}
\ds\prod_{m=1}^n \psi_m^V(z,\omega,0)}.
\label{eq:theta^*}
\end{equation}
\end{corollary}

\begin{proof}
By definition, we have
\[
\theta^*(z,\omega) 
=
\sum_{n=1}^{\infty} \theta_n(z,\omega).
\]
It then follows from (\ref{eq:theta_n}) and (\ref{eq:theta_1-org}) that
\begin{align*}
\theta^*(z,\omega) 
&=
\psi_1^H(z,\omega,0) \psi_1^V(z,\omega,0)
+
\sum_{n=2}^{\infty}
\psi_n^H(z,\omega,0) \prod_{m=1}^n \psi_m^V(z,\omega,0)
-
\sum_{n=2}^{\infty}
\sum_{k=1}^{n-1} \theta_k(z,\omega)
\prod_{m=1}^{n-k} \psi_m^V(z,\omega,0)
\\
&=
\sum_{n=1}^{\infty}
\psi_n^H(z,\omega,0) \prod_{m=1}^n \psi_m^V(z,\omega,0)
-
\sum_{k=1}^{\infty} \theta_k(z,\omega)
\sum_{n=k+1}^{\infty}
\prod_{m=1}^{n-k} \psi_m^V(z,\omega,0)
\\
&=
\sum_{n=1}^{\infty}
\psi_n^H(z,\omega,0) \prod_{m=1}^n \psi_m^V(z,\omega,0)
-
\theta^*(z,\omega)
\sum_{n=1}^{\infty}
\prod_{m=1}^n \psi_m^V(z,\omega,0),
\end{align*}
from which (\ref{eq:theta^*}) follows.
\end{proof}

\section{Applications to batch-service queues}

In this section, we consider the following batch-service M/G/1 queue.
Customers arrive according to a Poisson process with rate $\lambda$
and they are served by a single server. If the system is empty, the
server is idle and as soon as the next customer arrives, the sever
starts the service of this customer. If some customers wait for
service at the end of a service, the server serves all those customers
in batch, and otherwise, the sever becomes idle.  Note that if
multiple customers are served in batch, they leave the system
simultaneously at the end of the service. We assume that the batch
service time $S^{\rm B}$ of a batch with $k$ ($k=1,2,\ldots$)
customers is given by
\[
S^{\rm B} = H_1 + H_2 + \cdots + H_k + V.
\]
In \cite{Inoue21}, this batch-service M/G/1 queue is studied under the
assumption that $S^{\rm B}$ for a batch with $k$ customers is 
given by $S^{\rm B} = k\overline{H}+\overline{V}$, where 
$\overline{H}$ and $\overline{V}$ are constant.

It is easy to see that the PGF of the number of customers at the end
of a randomly chosen batch service in steady state is identical to
$\ell_E^*(z)$ in Theorem \ref{theorem:ell_E^*(z)}. Therefore, this
batch-service queue is stable if and only if $\rho < 1$ and $\E[V] <
\infty$. We define $L_S$ as the number of customers at the beginning
of a randomly chosen service in steady state and let $\ell_S^*(z)$
denote the PGF of $L_S$.  Note here that $L_S$ can also be considered
as a random variable representing the batch size (i.e., the number of
customers served simultaneously) in steady state.  It then follows
that
\begin{equation}
\ell_S^*(z)
=
\ell_E^*(z) + (z-1) \ell_E^*(0).
\label{eq:ell_S^*(z)}
\end{equation}
Let $C$ denote the length of the time interval between two consecutive
starts of batch services in steady state.  We then have
\begin{equation}
\E\big[S^{\rm B}\big] = \big(\E[L_E]+\ell_E^*(0)\big) \E[H]+\E[V],
\qquad
\E[C] = \E[S^{\rm B}] + \ell_E^*(0) \cdot \frac{1}{\lambda}.
\label{eq:ES-EC}
\end{equation}
We define $\rho^*$ as the utilization factor of the server.
\begin{equation}
\rho^* 
=
\frac{\E\big[S^{\rm B}\big]}{\E[C]}
=
\frac{ \rho \E[L_E] + \rho \ell_E^*(0) + \lambda \E[V]}
{\rho \E[L_E] + (1+\rho) \ell_E^*(0) + \lambda \E[V]}.
\label{eq:rho^*}
\end{equation}
Let $\widehat{S}^{\rm B}$ and $\widetilde{S}^{\rm B}$ denote the age
and the residual life of $S^{\rm B}$. We then define $f^*(z,s,\omega)$
as the joint transform of $L_S$, $\widehat{S}^{\rm B}$, and
$\widetilde{S}^{\rm B}$.
\begin{equation}
f^*(z,s,\omega)
=
\E\big[z^{L_S} e^{-s \widehat{S^{\rm B}}^{\rm B}} 
e^{-\omega \widetilde{S^{\rm B}}^{\rm B}}\big]
=
\frac{\ell_S^*\big(z h^*(\omega)\big)v^*(\omega)
- \ell_S^*\big(z h^*(s)\big) v^*(s)}
{\E[S^{\rm B}] (s-\omega)}.
\label{eq:f}
\end{equation}

\begin{theorem}
If $\rho < 1$ and $\E[V] < \infty$, the PGF $\ell^*(z)$ of the number
$L$ of customers in the stationary batch-service M/G/1 queue is given
by
\begin{align*}
\ell^*(z)
&=
1-\rho^* + \rho^* \cdot
\frac{\ell_S^*(z) - \ell_S^*\big(z h^*(\lambda-\lambda z)\big)
v^*(\lambda-\lambda z)}
{\lambda \E[S^{\rm B}] (1-z)},
\end{align*}
where $\ell_S^*(z)$, $\E[S^{\rm B}]$, and $\rho^*$ are given by
(\ref{eq:ell_S^*(z)}), (\ref{eq:ES-EC}), and (\ref{eq:rho^*}),
respectively.
\end{theorem}

\begin{proof}
It is easy to see that 
\[
L 
= 
\begin{cases}
0, & \mbox{with probability\ $1-\rho^*$},
\\
L_S + N_{\lambda}(\widehat{S}^{\rm B}),
& \mbox{with probability\ $\rho^*$},
\end{cases}
\]
and 
\[
f^*(z,s,0)
=
\E\big[z^{L_S} e^{-s \widehat{S}^{\rm B}} \big]
=
\frac{\ell_S^*(z)
- \ell_S^*\big(z h^*(s)\big) v^*(s)}
{\E[S^{\rm B}] s},
\]
from which, the theorem follows.
\end{proof}

Let $D$ denote the system delay of a randomly chosen tagged customer
in steady state.  

\begin{theorem}
If $\rho < 1$ and $\E[V] < \infty$, the LST $d^*(\omega)$ of the
system delay $D$ in steady state is given by
\[
d^*(\omega)
=
h^*(\omega) v^*(\omega)
\Big[
1-\rho^*
+
\rho^*
\cdot
\frac{\ell_E^*(h^*(\omega)) 
- 
\ell_S^*\big(h^*(\omega+\lambda-\lambda h^*(\omega))\big) 
v^*(\omega+\lambda-\lambda h^*(\omega))}
{\E[S^{\rm B}] \omega}
\Big],
\]
where $\ell_E^*(z)$ and $\ell_S^*(z)$ are given by
(\ref{eq:ell_E^*(z)}) and (\ref{eq:ell_S^*(z)}).
\end{theorem}

\begin{proof}
Because of the PASTA, the tagged customer finds the empty system on
arrival with probability $1-\rho^*$ and in this case,
$D=H+V$. Otherwise, $D$ is given by the sum of the residual service
time and the service time of the batch in which the tagged customer is
involved. Note here that the number of customers in the served batch
involving the tagged customer is given by
$N_{\lambda}(\widehat{S}^{\rm B}) + 1 + N_{\lambda}(\widetilde{S}^{\rm B})$. 
We thus have
\begin{align*}
D
&=
\begin{cases}
H+V, & \mbox{with probability\ $1-\rho^*$},
\\
\widetilde{S}^{\rm B}
+
\ds\sum_{n=1}^{N_{\lambda}(\widehat{S}^{\rm B})} H_n
+
H_{N_{\lambda}(\widehat{S}^{\rm B})+1}
+
\ds\sum_{n=N_{\lambda}(\widehat{S}^{\rm B})+2}^{N_{\lambda}(\widehat{S^{\rm B}})+1+
N_{\lambda}(\widetilde{S}^{\rm B})} H_n
+ V,
& \mbox{with probability\ $\rho^*$},
\end{cases}
\end{align*}
where the empty sum is defined as zero. We thus have
\[
d^*(\omega)
=
h^*(\omega) v^*(\omega)
\Big[
1-\rho^*
+
\rho^*
f^*\big(1,\lambda-\lambda h^*(\omega),\omega+\lambda-\lambda h^*(\omega)\big)
\Big],
\]
where $f^*(z,s,\omega)$ is given by (\ref{eq:f}) and by definition, 
\begin{align*}
f^*\big(1,\lambda-\lambda z,\omega+\lambda-\lambda z \big)
&=
\frac{\ell_S^*\big(h^*(\omega+\lambda-\lambda z)\big)
v^*(\omega+\lambda-\lambda z)
- \ell_S^*\big(h^*(\lambda-\lambda z)\big) v^*(\lambda-\lambda z)}
{\E[S^{\rm B}] (-\omega)}
\\
&=
\frac{\ell_E^*(z) 
- 
\ell_S^*\big(h^*(\omega+\lambda-\lambda z)\big) v^*(\omega+\lambda-\lambda z)}
{\E[S^{\rm B}] \omega},
\end{align*}
from which the theorem follows.
\end{proof}

\section{Remarks on gated-service MAP/GI/1 vacation queues}

We remark on the inapplicability of the analytical approach in this
paper to gated-service MAP/GI/1 vacation queues. Let $\bm{e}$ denote a
column vector of ones with an appropriate dimension. We assume that
customers arrive according to a MAP with representation
$(\bm{C},\bm{D})$, where $\bm{D} \geq \bm{O}$, $\bm{D}\neq\bm{O}$,
$(\bm{C}+\bm{D})\bm{e} = \bm{0}$, and $(-\bm{C})^{-1}\bm{D}$ is an
irreducible stochastic matrix.  Note that $\bm{C}+\bm{D}$ represents
the infinitesimal generator of a continuous-time Markov chain
$\{Y(t)\}$ on the finite state space $\M=\{1,2,\ldots,|\M|\}$, and
customers arrive when transitions driven by $\bm{D}$ occur.

As shown in Section \ref{section:busy_cycle}, as well as in
\cite{Takine92} for the multiple-vacation case, the queue length
process in gated-service M/GI/1 vacation queues is closely related to
the superposition of the \textit{independent} branching processes with
shifted generations. Note here that the independence of those
branching processes comes from the independent increment property of
Poisson arrivals. In the case of MAP arrivals, however, arrivals in
consecutive intervals are correlated, so that those branching
processes are correlated. As a result, the analytical approach in this
paper is not applicable to gated-service MAP/GI/1 vacation queues.

To clarify our claim, we consider a gated-service MAP/G/1 queue with
single vacations. Let $\tau_n$ denote the end of the $n$th vacation
and we define $\bm{q}_n^*(z)$ as
\[
\bm{q}_n^*(z) 
= 
\big(q_{n,1}^*(z)\ \ q_{n,2}^*(z)\ \ \cdots\ \ q_{n,|\M|}^*(z)\big),
\]
where 
\[
q_{n,i}^*(z) = \E\big[ z^{L_E(\tau_n)} \one_{\{Y(\tau_n)=i\}} \big],
\quad
i \in \M.
\]
It then follows that (cf.\ (\ref{eq:q_n=q_{n-1}}))
\begin{equation}
\bm{q}_n^*(z)
=
\bm{q}_{n-1}^*\big(\bm{A}^*(z)\big)\bm{V}^*(z)
-
\bm{q}_{n-1}^*(0) \big[\bm{I} - \bm{A}^*(z)\big] \bm{V}^*(z),
\quad
n=1,2,\ldots,
\label{eq:vec-q_n=q_{n-1}}
\end{equation}
where 
\[
\bm{A}^*(z) 
= 
\sum_{k=0}^{\infty} \bm{A}(k) z^k 
= 
\int_0^{\infty} \exp[(\bm{C}+z\bm{D})x] \dd H(x),
\qquad
\bm{V}^*(z) 
= 
\sum_{k=0}^{\infty} \bm{V}(k) z^k 
= 
\int_0^{\infty} \exp[(\bm{C}+z\bm{D})x]\dd V(x).
\]
Note that our approach requires that (\ref{eq:vec-q_n=q_{n-1}}) still holds
even if $\bm{A}^*(z)$ is substituted for $z$, i.e., 
\begin{equation}
\bm{q}_n^*\big(\bm{A}^*(z)\big)
=
\bm{q}_{n-1}^*\big(\bm{A}^*\big(\bm{A}^*(z)\big)\big)\bm{V}^*\big(\bm{A}^*(z)\big)
-
\bm{q}_{n-1}^*(0) \big[\bm{I} - \bm{A}^*\big(\bm{A}^*(z)\big)\big] 
\bm{V}^*\big(\bm{A}^*(z)\big),
\quad
n=1,2,\ldots.
\label{eq:vec2-q_n=q_{n-1}}
\end{equation}
Unfortunately, however, (\ref{eq:vec2-q_n=q_{n-1}}) does not hold in
general. See Appendix \ref{appendix:condition} for details.
Therefore, the results in \cite{Vishnevsky11} are not correct in
general.

\appendix

\section{Proof of Lemma \ref{lemma:FVT}}
\label{appendix:proof_lemma:FVT}

Let $b_n = a_n - a^*$ ($n=1,2,\ldots$). We then have for $|\omega| < 1$, 
\begin{align*}
\lim_{\omega \to 1-} (1-\omega) \sum_{n=1}^{\infty} a_n \omega^n
&=
\lim_{\omega \to 1-} (1-\omega) \sum_{n=1}^{\infty} (b_n + a^*) \omega^n
\\
&=
\lim_{\omega \to 1-} (1-\omega) \sum_{n=1}^{\infty} b_n \omega^n
+
\lim_{\omega \to 1-} 
a^* \omega
=
\lim_{\omega \to 1-} (1-\omega) \sum_{n=1}^{\infty} b_n \omega^n
+
a^* .
\end{align*}
Therefore, it suffices to show 
\begin{equation}
\lim_{\omega \to 1-} (1-\omega) \sum_{n=1}^{\infty} b_n \omega^n = 0
\label{eq:FVT=0}
\end{equation}
where 
\begin{equation}
\lim_{n \to \infty} b_n = \lim_{n \to \infty} (a_n-a^*) = 0.
\label{eq:b_n-limit}
\end{equation}
Note that (\ref{eq:b_n-limit}) is equivalent to 
\[
\forall \varepsilon' > 0,\ \exists N \geq 1,\ \forall n \geq N,\
|b_n| < \varepsilon',
\]
so that for all $z \in (0,1)$, 
\[
\bigg|
(1-z) \sum_{n=1}^{\infty} b_n z^n
\bigg|
<
(1-z) \sum_{n=1}^N |b_n| z^n + (1-z) \sum_{n=N+1}^{\infty} \varepsilon' z^n
<
(1-z) \sum_{n=1}^N |b_n| + \varepsilon'
=
(1-z) c + \varepsilon',
\]
where 
\[
c = \sum_{n=1}^N |b_n| \geq 0.
\]
Therefore, if $c=0$, (\ref{eq:FVT=0}) holds. We now assume
$c \neq 0$ and set 
$\varepsilon'=\varepsilon/2$, $\delta=1-\varepsilon/(2c)$.
It then follows that 
\[
\forall \varepsilon > 0,\ \ \forall z \in (\delta,1), \ \ 
\bigg|
(1-z) \sum_{n=1}^{\infty} b_n z^n
\bigg|
<
(1-z) c + \varepsilon'
<
(1-\delta) c + \varepsilon'
=
\frac{\varepsilon}{2c} \times c + \frac{\varepsilon}{2}
=\epsilon
\]
which is equivalent to (\ref{eq:FVT=0}).

\section{The condition under which (\ref{eq:vec2-q_n=q_{n-1}}) holds}
\label{appendix:condition}

We explore the condition under which
(\ref{eq:vec2-q_n=q_{n-1}}) holds.
For simplicity in description, we define 
\[
\bm{q}_n^*(z) 
=
\sum_{k=0}^{\infty} \bm{q}_n(k) z^k,
\qquad
\bm{f}_n^*(z) 
=
\sum_{k=0}^{\infty} \bm{f}_n(k) z^k 
=
\bm{q}_n^*\big(\bm{A}^*(z)\big),
\qquad
\bm{g}_n^*(z) 
=
\sum_{k=0}^{\infty} \bm{g}_n(k) z^k 
=
\bm{q}_n(0) \bm{V}^*(z).
\]
and rewrite (\ref{eq:vec-q_n=q_{n-1}}) to be
\begin{align*}
\sum_{k=0}^{\infty} \bm{q}_n(k) z^k
&=
\sum_{k=0}^{\infty} \bm{f}_{n-1}(k) z^k 
\times 
\sum_{k=0}^{\infty} \bm{V}(k) z^k 
-
\bm{g}_{n-1}(0) \bm{V}^*(z)
+
\bm{g}_{n-1}(0) 
\sum_{j=0}^{\infty} \bm{A}(j) z^j
\sum_{k=0}^{\infty} \bm{V}(k) z^k
\\
&=
\sum_{k=0}^{\infty} 
\bigg[
\sum_{j=0}^k \bm{f}_{n-1}(j) \bm{V}(k-j)
\bigg] z^k 
-
\bm{g}_{n-1}(0) \bm{V}^*(z)
+
\bm{g}_{n-1}(0) 
\sum_{k=0}^{\infty} 
\bigg[
\sum_{j=0}^k
\bm{A}(j) \bm{V}(k-j)
\bigg] z^k.
\end{align*}
The left-hand side of (\ref{eq:vec2-q_n=q_{n-1}}) is then given by
\begin{align}
\bm{q}_n^*\big(\bm{A}^*(z)\big)
&=
\sum_{k=0}^{\infty} 
\bigg[
\sum_{j=0}^k \bm{f}_{n-1}(j) \bm{V}(k-j)
\bigg] \big(\bm{A}^*(z)\big)^k
\notag
\\
&\qquad\qquad {} 
-
\bm{g}_{n-1}(0) \bm{V}^*\big(\bm{A}^*(z)\big)
+
\bm{g}_{n-1}(0) 
\sum_{k=0}^{\infty} 
\bigg[
\sum_{j=0}^k
\bm{A}(j) \bm{V}(k-j)
\bigg] \big(\bm{A}^*(z)\big)^k
\notag
\\
&=
\sum_{j=0}^{\infty} 
\bm{f}_{n-1}(j) 
\sum_{k=0}^{\infty}
\bm{V}(k)
\big(\bm{A}^*(z)\big)^j \big(\bm{A}^*(z)\big)^k
\notag
\\
&\qquad\qquad {}
-
\bm{g}_{n-1}(0) \bm{V}^*\big(\bm{A}^*(z)\big) 
+
\bm{g}_{n-1}(0) 
\sum_{j=0}^{\infty} 
\bm{A}(j) 
\sum_{k=0}^{\infty}
\bm{V}(k)
\big(\bm{A}^*(z)\big)^j \big(\bm{A}^*(z)\big)^k,
\label{eq:left}
\end{align}
while the right-hand side of (\ref{eq:vec2-q_n=q_{n-1}}) is given by
\begin{align}
\lefteqn{%
\bm{q}_{n-1}^*\big(\bm{A}^*\big(\bm{A}^*(z)\big)\big)\bm{V}^*\big(\bm{A}^*(z)\big)
-
\bm{q}_{n-1}^*(0) \big[\bm{I} - \bm{A}^*\big(\bm{A}^*(z)\big)\big] 
\bm{V}^*\big(\bm{A}^*(z)\big)
}\qquad\qquad\qquad
\notag
\\
&=
\bm{f}_{n-1}^*\big(\bm{A}^*(z))\big)\bm{V}^*\big(\bm{A}^*(z)\big)
-
\bm{q}_{n-1}^*(0) \bm{V}^*\big(\bm{A}^*(z)\big)
+
\bm{q}_{n-1}^*(0) \bm{A}^*\big(\bm{A}^*(z)\big)
\bm{V}^*\big(\bm{A}^*(z)\big)
\notag
\\
&=
\sum_{j=1}^{\infty} \bm{f}_{n-1}(j) \big(\bm{A}^*(z))\big)^j
\sum_{k=0}^{\infty} \bm{V}(k) \big(\bm{A}^*(z)\big)^k
\notag
\\
&\qquad\qquad {}
-
\bm{q}_{n-1}^*(0) \bm{V}^*\big(\bm{A}^*(z)\big)
+
\bm{q}_{n-1}^*(0) \sum_{j=0}^{\infty} \bm{A}(j) \big(\bm{A}^*(z)\big)^j
\sum_{k=0}^{\infty} \bm{V}(k) \big(\bm{A}^*(z)\big)^k.
\label{eq:right}
\end{align}
Comparing (\ref{eq:left}) and (\ref{eq:right}), we observe that 
(\ref{eq:vec2-q_n=q_{n-1}}) holds if 
\begin{equation}
\bm{A}^*(z) \bm{V}(k) = \bm{V}(k) \bm{A}^*(z) 
\quad \mbox{for all $k=0,1,\ldots$}.
\label{eq:condition}
\end{equation}
Note that (\ref{eq:condition}) does not hold in general.
For example, we assume $\bm{A}^*(z) \bm{V}(0) = \bm{V}(0) \bm{A}^*(z)$, 
i.e.,
\begin{equation}
\int_0^{\infty} \exp[(\bm{C}+z\bm{D})x] \dd H(x) 
\int_0^{\infty} \exp[\bm{C}y] \dd V(y) 
=
\int_0^{\infty} \exp[\bm{C}y] \dd V(y) 
\int_0^{\infty} \exp[(\bm{C}+z\bm{D})x] \dd H(x) .
\label{eq:condition_example}
\end{equation}
It is readily verified that (\ref{eq:condition_example}) holds if
$\bm{C}$ and $\bm{D}$ are commutative. 

Suppose $\bm{C}$ and $\bm{D}$ are commutative. We then have
\[
\bm{C}\bm{D} = \bm{D}\bm{C}\ \ \Leftrightarrow\ \ 
\bm{D}(-\bm{C})^{-1} = (-\bm{C})^{-1} \bm{D}, 
\]
so that
\[
\bm{D}(-\bm{C})^{-1}\bm{e} 
= 
(-\bm{C})^{-1} \bm{D} \bm{e}
=
(-\bm{C})^{-1} (-\bm{C}) \bm{e}
=
(-\bm{C}) (-\bm{C})^{-1} \bm{e},
\]
which is equivalent to 
\begin{equation}
(\bm{C}+\bm{D}) \times (-\bm{C})^{-1} \bm{e} = \bm{0}.
\label{eq:condition-1}
\end{equation}
Because $\bm{C}+\bm{D}$ is an infinitesimal generator of an irreducible, 
finite-state Markov chain, we have
\begin{equation}
(\bm{C}+\bm{D}) \bm{x} = \bm{0}\ \ \Rightarrow\ \ 
\bm{x} = \alpha \bm{e},
\label{eq:condition-2}
\end{equation}
where $\alpha$ is an arbitrarily constant.  It then follows from
(\ref{eq:condition-1}), $(-\bm{C})^{-1} \bm{e} > \bm{0}$, and
(\ref{eq:condition-2}) that there exists $\lambda > 0$ such that
\[
(-\bm{C})^{-1} \bm{e} = \frac{1}{\lambda} \bm{e}.
\]
Therefore, we have
\begin{equation}
\lambda \bm{e} = (-\bm{C})\bm{e} = \bm{D}\bm{e},
\label{eq:Poisson}
\end{equation}
which implies that the arrival rate is independent of the underlying
Markov chain $\{Y_t\}$.  We thus conclude that if $\bm{C}$ and
$\bm{D}$ are commutative, the MAP with representation
$(\bm{C},\bm{D})$ is reduced to a Poisson process with rate
$\lambda$. In fact, if $\bm{C}\bm{D} = \bm{D}\bm{C}$, we have
(\ref{eq:Poisson}) and
\begin{align*}
\bm{A}^*(z)\bm{e} 
=
\sum_{k=0}^{\infty} \bm{A}(k) \bm{e} z^k
&=
\int_0^{\infty} 
\exp[(\bm{C}+z\bm{D})x]\dd H(x) \bm{e}
=
\int_0^{\infty} 
\exp[(\bm{C}x] \sum_{k=0}^{\infty} \frac{z^k \bm{D}^k \bm{e} x^k}{k!} \dd H(x) 
\\
&=
\int_0^{\infty} 
\exp[(\bm{C}x] \bm{e} 
\sum_{k=0}^{\infty} \frac{z^k x^k \lambda^k}{k!} \dd H(x) 
=
\int_0^{\infty} 
\sum_{j=0}^{\infty} \frac{\bm{C}^j \bm{e} x^j}{j!}
\sum_{k=0}^{\infty} \frac{z^k \lambda^k x^k}{k!} \dd H(x) 
\\
&=
\int_0^{\infty} 
\sum_{j=0}^{\infty} \frac{(-\lambda)^j \bm{e} x^j}{j!}
\sum_{k=0}^{\infty} \frac{z^k \lambda^k x^k}{k!} \dd H(x) 
=
\int_0^{\infty} 
e^{-\lambda x} 
\sum_{k=0}^{\infty} \frac{\lambda^k x^k}{k!} \dd H(x) \bm{e} z^k ,
\end{align*}
so that 
\[
\bm{A}(k) \bm{e} 
= 
\int_0^{\infty} 
e^{-\lambda x} 
\frac{(\lambda x)^k}{k!} \dd H(x) \bm{e},
\quad
k=0,1,\ldots.
\]

\end{document}